\documentclass[a4paper,twoside,12pt,reqno]{article}
\usepackage[english]{babel}
\usepackage{graphicx}
\usepackage{epstopdf}
\usepackage[T1]{fontenc}
\usepackage[latin1]{inputenc}
\usepackage{amssymb,amsmath,amsthm,dsfont,mathrsfs}
\usepackage{amsfonts,euscript}
\usepackage{color}
\usepackage{authblk}
\usepackage{multirow}
\usepackage{mathtools}

\usepackage{authblk}
\usepackage{fullpage}
\usepackage{setspace}

\newtheorem{teo}{Theorem}[section]
\newtheorem{pro}[teo]{Proposition}

\newtheorem{ex}{Example}[section]

\begin{document}

\renewcommand{\baselinestretch}{2}	
	
	\title{The stopped clock model}

	\author{Helena Ferreira}
	\affil{Universidade da Beira Interior, Centro de Matem\'{a}tica e Aplica\c{c}\~oes (CMA-UBI), Avenida Marqu\^es d'Avila e Bolama, 6200-001 Covilh\~a, Portugal\\ \texttt{helena.ferreira@ubi.pt}}
	
	\author{Marta Ferreira}
	\affil{Center of Mathematics of Minho University\\ Center for Computational and Stochastic Mathematics of University of Lisbon\\
		Center of Statistics and Applications of University of Lisbon, Portugal\\ \texttt{msferreira@math.uminho.pt} }

	\date{}
	
	\maketitle

\abstract{The extreme values theory presents specific tools for modeling and predicting extreme phenomena. In particular, risk assessment is often analyzed through measures for tail dependence and high values clustering. Despite technological advances allowing an increasingly larger and more efficient data collection, there are sometimes failures in the records, which causes difficulties in statistical inference, especially in the tail where data are scarcer. In this article we present a model with a simple and intuitive failures scheme, where each record failure is replaced by the last record available. We will study its extremal behavior with regard to local dependence and high values clustering, as well as the temporal dependence on the tail.}

\bigskip

\noindent\textbf{keywords:} {extreme values; stationary sequences; failures model; extremal index; tail dependence coefficient.}\\


\section{Introduction}\label{sintro}


 Let $\{X_n\}_{n\in\mathbb{Z}}$ and $\{U_n\}_{n\in\mathbb{Z}}$ be stationary sequences of real random variables  on the probability space $(\Omega, {\cal{A}}, P)$ and $P(U_n\in\{0,1\})=1$. We  define, for $n\geq 1$,
\begin{eqnarray}\label{Yn}
Y_n=\left\{
\begin{array}{ll}
X_n & ,\, U_n=1\\
Y_{n-1} & ,\, U_n=0\,.
\end{array}
\right.
\end{eqnarray}
Sequence $\{Y_n\}_{n\geq 1}$ corresponds to a model of failures on records of $\{X_n\}_{n\in\mathbb{Z}}$ replaced by the last available record, which occurs in some random past instant, if we interpret $n$ as time. Thus, if for example it occurs $\{U_1=1,U_2=0,U_3=1,U_4=0,U_5=0,U_6=0,U_7=1\}$, we will have $\{Y_1=X_1,Y_2=X_1,Y_3=X_3,Y_4=X_3,Y_5=X_3,Y_6=X_3,Y_7=X_7\}$. This constancy of some variables of $\{X_n\}_{n\in\mathbb{Z}}$ for random periods of time motivates the designation of ``stopped clock model" for sequence $\{Y_n\}_{n\geq 1}$. 

Failure models studied in the literature from the point of view of extremal behavior do not consider the stopped clock model (Hall and H\"usler, \cite{HallHusler2006} 2006; Ferreira \emph{et al.}, \cite{FerreiraH++2019} 2019 and references therein). 

The model we will study can also be represented by $\{X_{N_n}\}_{n\geq 1}$ where $\{N_n\}_{n\geq 1}$ is a sequence of positive integer variables representable by 
\begin{eqnarray}\nonumber
N_n=nU_n+\sum_{i\geq 1}\left(\prod_{j=0}^{i-1}(1-U_{n-j})\right)U_{n-i}(n-i),\,n\geq 1\,.
\end{eqnarray}

We can also state a recursive formulation for $\{Y_n\}_{n\geq 1}$ through
\begin{eqnarray}\nonumber
Y_n=X_nU_n+\sum_{i\geq 1}\left(\prod_{j=0}^{i-1}(1-U_{n-j})\right)U_{n-i}X_{n-i}+\prod_{i\geq 0}(1-U_{n-i})Y_{n-\kappa},\,n\geq 1,\,\,\,\kappa\geq 1\,.
\end{eqnarray}

Under any of the three possible representations (failures model, random index sequence or recursive sequence), we are not aware of an extremal behavior study of $\{Y_n\}_{n\geq 1}$ in the literature.

Our departure hypotheses about the base sequence $\{X_n\}_{n\in\mathbb{Z}}$ and about sequence $\{U_n\}_{n\in\mathbb{Z}}$ are:

\begin{enumerate}
	\item[(1)] $\{X_n\}_{n\in\mathbb{Z}}$ is a stationary sequence of random variables almost surely distinct and, without loss of generality, such that $F_{X_n}(x):=F(x)=\exp(-1/x)$, $x>0$, i.e., standard Fr\'echet distributed.
	\item[(2)] $\{X_n\}_{n\in\mathbb{Z}}$ and $\{U_n\}_{n\in\mathbb{Z}}$ are independent.
	\item[(3)] $\{U_n\}_{n\in\mathbb{Z}}$ is stationary and $p_{n_1,...,n_s}(i_1,...,i_s):=P(U_{n_1}=i_{1},...,U_{n_s}=i_s)$, $i_j\in\{0,1\}$, $j=1,...,s$, is such that $p_{n,n+1,...,n+\kappa-1}(0,...,0)=0$, for some $\kappa\geq 1$.
\end{enumerate}
The trivial case $\kappa=1$ corresponds to $Y_n=X_n$, $n\geq 1$. Hypothesis (3) means that we are assuming that it is almost impossible to lose $\kappa$ or more consecutive values of $\{X_n\}_{n\in\mathbb{Z}}$.
We remark that, along the paper, the summations, produts and  intersections is considered to be non-existent whenever the end of the counter is less than the beginning. We will also use notation $a\vee b=\max(a,b)$. 

\begin{ex}\label{ex1}
	Consider an independent and identically distributed sequence $\{W_n\}_{n\in\mathbb{Z}}$ of real random variables on  $(\Omega, {\cal{A}}, P)$ and a Borelian set $A$. Let $p=P(A_n)$ where $A_n=\{W_n\in A\}$, $n\in\mathbb{Z}$.
The sequence of Bernoulli random variables 
\begin{eqnarray}\label{exUn}
U_n=\mathbf{1}_{\{\bigcap_{i=1}^{\kappa-1}\overline{A}_{n-i}\}}+(1-\mathbf{1}_{\{\bigcap_{i=1}^{\kappa-1}\overline{A}_{n-i}\}})\mathbf{1}_{\{A_{n}\}},\,{n\in\mathbb{Z}},
\end{eqnarray}
where $\mathbf{1}_{\{\cdot\}}$ denotes the indicator function, defined for some fixed $\kappa\geq 2$,
 is such that $p_{n,n+1,...,n+\kappa-1}(0,...,0)=0$, i.e., it is almost sure that after $\kappa-1$ consecutive variables equal to zero, the next variable takes value one. 
In fact,  for any choice of $\kappa \geq 2$, 
\begin{eqnarray}\nonumber
\begin{array}{rl}
p_{n,n+1,...,n+\kappa-1}(0,...,0)&=P(U_n=0=U_{n+1}=...=U_{n+\kappa-1}) \\
&\leq P\left(\{\displaystyle\bigcap_{i=n}^{n+\kappa-2}\overline{A}_{i}\} \cap  \{U_{n+\kappa-1}=0\}\right)\\
&=P\left(\mathbf{1}_{\{\bigcap_{i=1}^{\kappa-1}\overline{A}_{n+\kappa-1-i}\}}=1, \{U_{n+\kappa-1}=0\}\right)=0.
\end{array}
\end{eqnarray}
We also have 
\begin{eqnarray}\nonumber
\begin{array}{rl}
p_n(0)=&P\left(\mathbf{1}_{\{\bigcap_{i=1}^{\kappa-1}\overline{A}_{n-i}\}}=0,\mathbf{1}_{\{A_{n}\}}=0\right)=P(\bigcup_{i=1}^{\kappa-1}A_{n-i}\cap\overline{A}_n)\\
=&P(\overline{A}_n)-P(\bigcap_{i=0}^{\kappa-1}\overline{A}_{n-i})=1-p-(1-p)^\kappa,
\end{array}
\end{eqnarray}
since the independence of random variables $W_n$ implies the independence of events $A_n$, and, for $\kappa>2$,
\begin{eqnarray}\nonumber
\begin{array}{rl}
p_{n-1,n}(0,0)=&P((\bigcup_{i=1}^{\kappa-1}A_{n-i}\cap\overline{A}_n)\cap (\bigcup_{i=1}^{\kappa-1}A_{n-1-i}\cap\overline{A}_{n-1}))\\
=&P(\overline{A}_n\cap\overline{A}_{n-1})-P(\overline{A}_n\cap\overline{A}_{n-1}\cap(\bigcap_{i=1}^{\kappa-1}\overline{A}_{n-i}\cup \bigcap_{i=1}^{\kappa-1}\overline{A}_{n-1-i}))\\
=&P(\overline{A}_n)P(\overline{A}_{n-1})-P(\bigcap_{i=0}^{\kappa-1}\overline{A}_{n-i})=(1-p)^2-(1-p)^\kappa,
\end{array}
\end{eqnarray}
$p_{n-1,n}(1,0)=p_n(0)-p_{n-1,n}(0,0)=p(1-p).$

In Figure \ref{FigYn} we illustrate with a particular example based on independent standard Fr\'echet $\{X_n\}_{n\in\mathbb{Z}}$,   $\{W_n\}_{n\in\mathbb{Z}}$ with standard exponential marginals, $A=]0,1/2]$ and thus $p=0.3935$ and considering $\kappa=3$. Therefore, $p_{n,n+1,n+2}(0,0,0)=0$, $p_{n,n+1,n+2}(1,0,0)=p_{n,n+1}(0,0)=p(1-p)^2$.

\begin{figure}
	\centering
	\includegraphics[width=7cm,height=7cm]{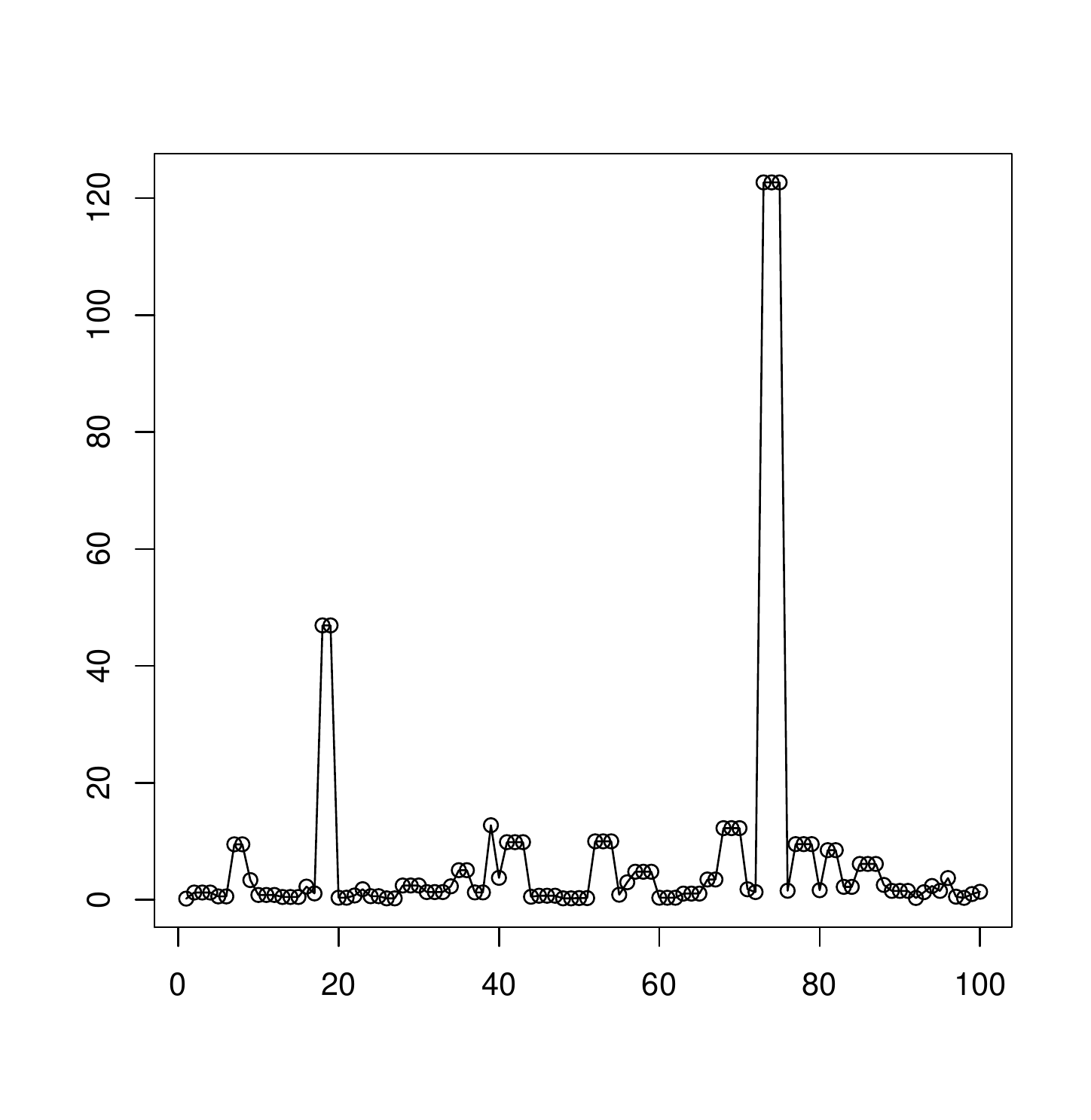}
	\caption{Sample path of $100$ observations simulated from $\{Y_n\}$ defined in (\ref{Yn}) based on independent standard Fr\'echet $\{X_n\}$ and on $\{U_n\}$ given in (\ref{exUn}) where we take random variables $\{W_n\}$ standard exponential distributed, $A=]0,1/2]$ and thus $p=0.3935$ and considering $\kappa=3$.\label{FigYn}}
\end{figure}
\end{ex}
\vspace{0,5cm}

In the next section we propose an estimator for probabilities $p_{n,...,n+s}(1,0,...,0)$, $0\leq s<\kappa-1$. In Section \ref{sextremalindex} we analyse the existence of the extremal index for $\{Y_n\}_{n\geq 1}$, an important measure to evaluate the tendency to occur clusters of its high values. A characterization of the tail dependence will be presented in Section \ref{stdc}. The results are illustrated with an ARMAX sequence.

For the sake of simplicity, we will omit the variation of $n$ in sequence notation whenever there is no doubt, taking into account that  we will keep the designation $\{Y_n\} $ for the stopped clock model and $\{X_n\}$ and $\{U_n\}$ for the sequences that generate it.

\section{Inference on $\{U_n\} $}\label{sinfer}

Assuming that $\{U_n\} $ is not observable, as well as the values of $\{X_n\} $ that are lost, it is of interest to retrieve information about these sequences from the available sequence $\{Y_n\} $. 

Since, for $n\geq 1$ and $s\geq 1$, we have
\begin{eqnarray}\nonumber
\begin{array}{rl}
&\displaystyle p_n(1)=E\left(\mathbf{1}_{\{Y_n\not = Y_{n-1}\}}\right),\,\, p_n(0)=E\left(\mathbf{1}_{\{Y_n = Y_{n-1}\}}\right)\\
 \textrm{ and }& \displaystyle p_{n-s,n-s+1,...,n}(1,0,...,0)=E\left(\mathbf{1}_{\{Y_{n-s-1}\not =Y_{n-s}=Y_{n-s+1}=...= Y_{n}\}}\right),
\end{array}
\end{eqnarray}
we propose to estimate these probabilities from the respective empirical counterparts of a random sample $(\hat{Y}_1, \hat{Y}_2,...,\hat{Y}_m)$ from $\{Y_n\} $, i.e.,  
\begin{eqnarray}\nonumber
\begin{array}{rl}
&\displaystyle \widehat{p}_n(1)=\frac{1}{m}\sum_{i=2}^{m}\mathbf{1}_{\{\hat{Y}_i\not = \hat{Y}_{i-1}\}},\,\, \widehat{p}_n(0)=\frac{1}{m}\sum_{i=2}^{m}\mathbf{1}_{\{\hat{Y}_i = \hat{Y}_{i-1}\}}\\
\textrm{ and }& \displaystyle \widehat{p}_{n-s,n-s+1,...,n}(1,0,...,0)=\frac{1}{m}\sum_{i=s+2}^{m}\mathbf{1}_{\{\hat{Y}_{i-s-1}\not =\hat{Y}_{i-s}=\hat{Y}_{i-s+1}=...= \hat{Y}_{i}\}}\,,
\end{array}
\end{eqnarray}
which are consistent by the {\it weak law of large numbers}. 
The value of $\kappa$ can be inferred from 
\begin{eqnarray}\nonumber
\widehat{\kappa}=\bigvee_{i=s+2}^{m}\,\,\bigvee_{s\geq 1} s \,\,\mathbf{1}_{\{\hat{Y}_{i-s-1}\not =\hat{Y}_{i-s}=\hat{Y}_{i-s+1}=...= \hat{Y}_{i}\}}\,.
\end{eqnarray}

In order to evaluate the finite sample behavior of the estimators above, we have simulated $1000$ independent replicas with size $m=100,1000,5000$ of the model in Example \ref{ex1}. The absolute bias (abias) and root mean squared error (rmse) are presented in Table \ref{tab1}. The results reveal a good performance of the estimators, even in the case of smaller sample sizes. Parameter $\kappa$ was always estimated with no error. 

\begin{table}
	\caption{The absolute bias (abias) and root mean squared error (rmse) obtained from 1000 simulated samples with size $m=100,1000,5000$ of the model in Example \ref{ex1}. \label{tab1}}	
\begin{center}	
	
	\begin{tabular}{ll|cc}
		
		& & abias & rmse\\
		\hline
 & $m=100$ & 0.0272 & 0.0335 \\
$\widehat{p}_n(0)$ & $m=1000$ & 0.0087 & 0.0108 \\
 & $m=5000$ & 0.0039 & 0.0048 \\
\hline
 & $m=100$ & 0.0199 & 0.0253 \\
$\widehat{p}_{n-1,n}(1,0)$ & $m=1000$ & 0.0065 & 0.0080 \\
 & $m=5000$ & 0.0030 & 0.0037 \\
\hline
 & $m=100$ & 0.0160 & 0.0200\\
$\widehat{p}_{n-2,n-1,n}(1,0,0)$ & $m=1000$ & 0.0051 & 0.0064\\
 & $m=5000$ & 0.0022 & 0.0028\\
	\end{tabular}
\end{center}
\end{table}

\section{The extremal index of $\{Y_n\} $}\label{sextremalindex}

The sequence $\{Y_n\} $ is stationary because the sequences $\{X_n\} $ and $\{U_n\} $ are stationary and independent from each other. In addition, the common distribution for $Y_n$, $n \geq 1$, is also standard Fr\'echet, as is the common distribution for $X_n$, since
\begin{eqnarray}\nonumber
\begin{array}{rl}
F_{Y_n}(x)=&\displaystyle \sum_{i=1}^{\kappa-1}P(X_{n-i}\leq x,U_{n-i}=1,U_{n-i+1}=0=...=U_n)+P(X_n\leq x)P(U_n=1)\\
=& \displaystyle F(x)\left(p_n(1)+\sum_{i=1}^{\kappa-1}p_{n-i,...,n}(1,0,...,0)\right)=F(x)\,.
\end{array}
\end{eqnarray}
For any $\tau>0$, if we define $u_n\equiv u_n(\tau)=n/\tau$, $n\geq 1$,
it turns out that $E\left(\sum_{i=1}^{n}\mathbf{1}_{\{Y_i>u_n\}}\right)=nP(Y_1>u_n)
\displaystyle\mathop{\longrightarrow}_{n\to\infty}\tau$ and  $nP(X_1>u_n)\displaystyle\mathop{\longrightarrow}_{n\to\infty}\tau$, so we refer to these levels $u_n$ by {\it normalized levels} for $\{Y_n\} $ and $\{X_n\} $.

In this section, in addition to the general assumptions about the model presented in Section \ref{sintro}, we start by assuming that $\{X_n\}$ and $\{U_n\}$ present dependency structures such that variables sufficiently apart can be considered approximately independent. Concretely, we assume that $\{U_n\} $ satisfies the strong-mixing condition (Rosenblat \cite{Rosenblat1956} 1956) and $\{X_n\} $ satisfies condition $D(u_n)$ (Leadbetter \cite{Leadbetter1974} 1974) for normalized levels $u_n$.

\begin{pro}\label{pD}
	If $\{U_n\} $ is strong-mixing and $\{X_n\} $ satisfies condition $D(u_n)$ then $\{Y_n\} $ also satisfies condition $D(u_n)$.
\end{pro}
\begin{proof}
	For any choice of $p+q$ integers, $1\leq i_1<...<i_p<j_1<...<j_q\leq n$ such that $j_1\geq i_p+l$, we have that
	\begin{eqnarray}\nonumber
	\left|P\left(\bigcap_{s=1}^p X_{i_s}\leq u_n,\bigcap_{s=1}^q X_{j_s}\leq u_n\right)-P\left(\bigcap_{s=1}^p X_{i_s}\leq u_n\right)P\left(\bigcap_{s=1}^q X_{j_s}\leq u_n\right)\right|\leq \alpha_{n,l},
	\end{eqnarray}
	with $\alpha_{n,l_n}\to 0$, as $n\to\infty$, for some sequence $l_n=o(n)$, and
		\begin{eqnarray}\nonumber
	\left|P\left(A\cap B\right)-P\left(A\right)P\left(B\right)\right|\leq g(l),
	\end{eqnarray}
	with $g(l)\to 0$, as $l\to\infty$, where $A$ belongs to the $\sigma$-algebra generated by $\{U_i,\,i=1,...,i_p\}$ and $B$ belongs to the $\sigma$-algebra generated by $\{U_i,\,i=j_1,j_1+1,...\}$. Thus, for any choice of $p+q$ integers, $1\leq i_1<...<i_p<j_1<...<j_q\leq n$ such that $j_1\geq i_p+l+\kappa$, we will have 
	\begin{eqnarray}\nonumber
	\begin{array}{rl}
	&\displaystyle \left|P\left(\bigcap_{s=1}^p Y_{i_s}\leq u_n,\bigcap_{s=1}^q Y_{j_s}\leq u_n\right)-P\left(\bigcap_{s=1}^p Y_{i_s}\leq u_n\right)P\left(\bigcap_{s=1}^q Y_{j_s}\leq u_n\right)\right|\\\\
	\leq & \displaystyle \sum_{\substack{i_s-\kappa<i^*_s\leq i_s\\ j_s-\kappa<j^*_s\leq j_s}}\left|P\left(\bigcap_{s=1}^p X_{i_s^*}\leq u_n,\bigcap_{s=1}^q X_{j_s^*}\leq u_n\right)P\left(A^*\cap B^*\right)\right.\\
	&\hspace{1.8cm}\displaystyle \left.-P\left(\bigcap_{s=1}^p X_{i_s^*}\leq u_n\right)P\left(\bigcap_{s=1}^q X_{j_s^*}\leq u_n\right)P\left(A^*\right)P\left(B^*\right)\right|,
	\end{array}
	\end{eqnarray}
	where $A^*=\bigcap_{s=1}^p \{U_{i_s}=0=...=U_{i^*_s+1},U_{i^*_s}=1\}$ and $B^*=\bigcap_{s=1}^q \{U_{j_s}=0=...=U_{j^*_s+1},U_{j^*_s}=1\}$	and $j^*_1>j_1-\kappa\geq i^*_p+l$. Therefore, the last summation above is upper limited by
	\begin{eqnarray}\nonumber
	\begin{array}{rl}
	& \displaystyle \sum_{\substack{i_s-\kappa<i^*_s\leq i_s\\ j_s-\kappa<j^*_s\leq j_s}}\Bigg(\left|P\left(\bigcap_{s=1}^p X_{i_s^*}\leq u_n,\bigcap_{s=1}^q X_{j_s^*}\leq u_n\right)-P\left(\bigcap_{s=1}^p X_{i_s^*}\leq u_n\right)P\left(\bigcap_{s=1}^q X_{j_s^*}\leq u_n\right)\right|\\
	&\hspace{1.8cm}\displaystyle + |P\left(A^*\cap B^*\right)- P\left(A^*\right)P\left(B^*\right)|\Bigg)\\\\
\leq& \displaystyle \sum_{\substack{i_s-\kappa<i^*_s\leq i_s\\ j_s-\kappa<j^*_s\leq j_s}}\left(\alpha_{n,l}+g(l)\right),
	\end{array}
	\end{eqnarray}
	which allows to conclude that $D(u_n)$ holds for $\{Y_n\}$ with $l_n^{(Y)}=l_n+\kappa$.
\end{proof}

The tendency for clustering of values of $\{Y_n\}$ above $u_n$ depends on the same tendency within $\{X_n\}$ and the propensity of $\{U_n\}$ for consecutive null values. The clustering tendency can be assessed through the extremal index (Leadbetter, \cite{Leadbetter1974} 1974). More precisely, $\{X_n\}$ is said to have extremal index $\theta_X\in (0,1]$ if
\begin{eqnarray}\label{thetaDef}
\lim_{n\to\infty}P\left(\bigvee_{i=1}^{n}X_i\leq n/\tau\right)=e^{-\theta_X\tau}.
\end{eqnarray}

If $D(u_n)$ holds for $\{X_n\}$, we have
\begin{eqnarray}\nonumber
\lim_{n\to\infty}P\left(\bigvee_{i=1}^{n}X_i\leq u_n\right)=\lim_{n\to\infty}P^{k_n}\left(\bigvee_{i=1}^{[n/k_n]}X_i\leq u_n\right),
\end{eqnarray}
for any integers sequence $\{k_n\}$, such that, 
\begin{eqnarray}\label{kn}
k_n\to\infty,\,\,k_nl_n/n\to 0 \textrm{ and } k_n\alpha_{n,l_n}\to 0, \textrm{ as } n\to\infty. 
\end{eqnarray}
We can therefore say that
\begin{eqnarray}\nonumber
\theta_X\tau=\lim_{n\to\infty} k_n P\left(\bigvee_{i=1}^{[n/k_n]}X_i> u_n\right)\,.
\end{eqnarray}
Now we compare the local behavior of sequences $\{X_n\}$ and $\{Y_n\}$, i.e., of $X_i$ and $Y_i$ for $i\in\left\{(j-1)\left[\frac{n}{k_n}\right]+1,...,j\left[\frac{n}{k_n}\right]\right\}$, $j=1,...,k_n$, with regard to the oscillations of their values in relation to $u_n$. To this end, we will use local dependency conditions $D^{(s)}(u_n)$. We say that $\{X_n\}$ satisfies $D^{(s)}(u_n)$, $s\geq 2$, whenever
\begin{eqnarray}\nonumber
\lim_{n\to\infty} n \sum_{j=s}^{[n/k_n]} P\left(X_1> u_n, X_j\leq u_n<X_{j+1}\right)=0,
\end{eqnarray}
for some integers sequence  $\{k_n\}$ satisfying (\ref{kn}). Condition $D^{(1)}(u_n)$ translates into
\begin{eqnarray}\nonumber
\lim_{n\to\infty} n \sum_{j=2}^{[n/k_n]} P\left(X_1> u_n, X_{j}> u_n\right)=0,
\end{eqnarray}
and is known as condition $D^{'}(u_n)$ (Leadbetter \emph{et al.}, \cite{Leadbetter+1983} 1983), related to a unit extremal index, i.e., absence of  extreme values clustering. In particular, this is the case of independent variables. Although $\{X_n\}$ satisfies $D^{'}(u_n)$, this condition is not generally valid for $\{Y_n\}$. Observe that
\begin{eqnarray}\nonumber
\begin{array}{rl}
&\displaystyle n \sum_{j=2}^{[n/k_n]} P\left(Y_1> u_n, Y_{j}> u_n\right)\\\\
=&\displaystyle \sum_{i=2-\kappa}^{1} n \sum_{j=2}^{[n/k_n]}\sum_{j^*=i\vee (j-\kappa+1)}^{j} P\left(X_i> u_n, X_{j^*}> u_n\right)\cdot\\
&\hspace{4cm}\cdot p_{i,...,1,j^*,j^*+1,...,j}(0,...,0,1,0,...,0)\,.
\end{array}
\end{eqnarray}
For $i=1$ and $j=\kappa$, we have $j^*=1$ and the corresponding term becomes $nP(X_1>u_n)\to\tau>0$, as $n\to\infty$, reason why, in general $\{Y_n\}$ does not satisfy $D^{'}(u_n)$ even if $\{X_n\}$ satisfies it.

\begin{pro}\label{pD's}
	The following statements hold:
	\begin{itemize}
		\item [(i)] If $\{Y_n\}$ satisfies $D^{(s)}(u_n)$, $s\geq 2$, then $\{X_n\}$ satisfies $D^{(s)}(u_n)$.
		\item [(ii)] If $\{X_n\}$ satisfies $D^{(s)}(u_n)$, $s\geq 2$, then $\{Y_n\}$ satisfies $D^{(s+\kappa-1)}(u_n)$.
		\item [(iii)] If $\{X_n\}$ satisfies $D^{'}(u_n)$, then $\{Y_n\}$ satisfies $D^{(2)}(u_n)$.
	\end{itemize}
\end{pro}

\begin{proof}
	Consider $r_n=[n/k_n]$. We have that
\begin{eqnarray}\label{pD'sExp1}
\begin{array}{rl}
&\displaystyle n \sum_{j=s}^{r_n} P\left(Y_1> u_n, Y_{j}\leq u_n<Y_{j+1}\right)
\\\\
=&\displaystyle \sum_{i=2-\kappa}^{1} n \sum_{j=s}^{r_n}P\left(X_i> u_n, Y_{j}\leq u_n<X_{j+1},U_i=1,U_{i+1}=0=...=U_1,U_{j+1}=1\right)\\\\
=&\displaystyle \sum_{i=2-\kappa}^{1} n \sum_{j=s}^{r_n}\sum_{j^*=(i+1)\vee (j-\kappa+1)}^{j} P\left(X_i> u_n, X_{j^*}\leq u_n<X_{j+1}\right)\cdot\\
&\hspace{4cm}\cdot p_{i,i+1,...,1,j^*,j^*+1,...,j,j+1}(1,0,...,0,1,0,...,0,1)\,.
\end{array}
\end{eqnarray}	
Since $\{Y_n\}$ satisfies $D^{(s)}(u_n)$, with $s\geq 2$, and thus the first summation in (\ref{pD'sExp1}) converges to zero, as $n\to\infty$, then all the terms in the last summations also converge to zero. 
In particular, when $i=1$ and $j^*=j$, we have  $n \sum_{j=s}^{r_n} P\left(X_1> u_n, X_{j}\leq u_n<X_{j+1}\right)\to 0$, as $n\to\infty$, which proves (i).

On the other hand, writing the first summation in (\ref{pD'sExp1}) with $j$ starting at $s+\kappa-1$, we have
\begin{align}
&\displaystyle n \sum_{j=s+\kappa-1}^{r_n} P\left(Y_1> u_n, Y_{j}\leq u_n<Y_{j+1}\right)
\nonumber\\
=&\displaystyle \sum_{i=2-\kappa}^{1} n \sum_{j=s+\kappa-1}^{r_n}\sum_{j^*=j-\kappa+1}^{j} P\left(X_i> u_n, X_{j^*}\leq u_n<X_{j+1}\right)\cdot\nonumber\\
&\hspace{4cm}\cdot p_{i,i+1,...,1,j^*,j^*+1,...,j,j+1}(1,0,...,0,1,0,...,0,1)\nonumber\\
=&\displaystyle \sum_{i=2-\kappa}^{1} n \sum_{j=s+\kappa-1}^{r_n}\sum_{j^*=j-\kappa+1}^{j}\sum_{i^*=j^*}^{j} P\left(X_i> u_n, X_{j^*}\leq u_n,...,X_{i^*}\leq u_n,X_{j+1}>u_n\right)\nonumber\cdot\\
&\hspace{4cm}\cdot p_{i,i+1,...,1,j^*,j^*+1,...,j,j+1}(1,0,...,0,1,0,...,0,1)\label{pD'sExp2}\,,
\end{align}	
where the least of distances between $i$ and $i^*$ corresponds to the case $i=1$ and $i^*=j^*=s$. Therefore, if $\{X_n\}$ satisfies $D^{(s)}(u_n)$ for some $s\geq 2$ then each term of (\ref{pD'sExp2}) converges to zero, as $n\to\infty$, and thus  $\{Y_n\}$ satisfies $D^{(s+\kappa-1)}(u_n)$, proving (ii).

As for (iii), observe that
\begin{align}
&\displaystyle n \sum_{j=2}^{r_n} P\left(Y_1> u_n, Y_{j}\leq u_n<Y_{j+1}\right)
\nonumber\\
=&\displaystyle \sum_{i=2-\kappa}^{1} n \sum_{j=2}^{r_n}P\left(X_i> u_n, Y_{j}\leq u_n<X_{j+1},U_i=1,U_{i+1}=0=...=U_1,U_{j+1}=1\right)\nonumber\\
\leq &\displaystyle \sum_{i=2-\kappa}^{1} n \sum_{j=2}^{r_n} P\left(X_i> u_n, X_{j+1}>u_n\right)
=\sum_{i=2-\kappa}^{1} n \sum_{j=2}^{r_n} P\left(X_1> u_n, X_{j-i+2}>u_n\right)\nonumber\\
\leq &\displaystyle \,\,\kappa\, n \sum_{j=2}^{r_n} P\left(X_1> u_n, X_{j}>u_n\right)\label{pD'sExp3}\,.
\end{align}
If $\{X_n\}$ satisfies $D^{'}(u_n)$, then (\ref{pD'sExp3}) converges to zero, as $n\to\infty$, and $D^{(2)}(u_n)$ holds for $\{Y_n\}$.
\end{proof}

Under conditions $D(u_n)$ and $D^{(s)}(u_n)$ with $s\geq 2$, we can also compute the extremal index $\theta_X$ defined in (\ref{thetaDef}) by (Chernick \emph{et al.}, \cite{Chernick+1991} 1991; Corollary 1.3) 
\begin{eqnarray}\label{thetaChernick}
\theta_X =\displaystyle \lim_{n\to\infty}P\left(X_{2}\leq u_n,...,X_{s}\leq u_n|X_1>u_n\right)\,.
\end{eqnarray}

If $\{X_n\}$ and $\{Y_n\}$ have extremal indexes $\theta_X$ and $\theta_Y$, respectively, then $\theta_Y\leq \theta_X$, since $P(\bigvee_{i=1}^{n}X_i\leq n/\tau)\leq P(\bigvee_{i=1}^{n}Y_i\leq n/\tau)$. This corresponds to the intuitively expected, if we remember that the possible repetition of variables $X_n$ leads to larger clusters of values above $u_n$. In the following result, we establish a relationship between $\theta_X$ and $\theta_Y$.

\begin{pro}\label{pthetaYX}
	Suppose that $\{U_n\}$ is strong-mixing and $\{X_n\}$ satisfies conditions $D(u_n)$ and $D^{(s)}(u_n)$, $s\geq 2$, for normalized levels $u_n\equiv u_n(\tau)$. If $\{X_n\}$ has extremal index $\theta_X$ then $\{Y_n\}$ has extremal index $\theta_Y$ given by
	\begin{eqnarray}\nonumber
	\theta_Y=\theta_X\,\sum_{j=0}^{\kappa-1} p_{1,2,...,j+1,j+2}(1,0,...,0,1)\,\beta_j,
	\end{eqnarray}
	where 
	\begin{eqnarray}\nonumber
	\beta_j =\lim_{n\to\infty} P(X_{s+j}>u_n|X_1\leq u_n,...,X_{s-1}\leq u_n<X_s)\,.
	\end{eqnarray}
\end{pro}
\begin{proof}
By Proposition \ref{pD}, $\{Y_n\}$ also satisfies condition $D(u_n)$. Thus we have
\begin{eqnarray}\nonumber
\lim_{n\to\infty} P\left(\bigvee_{i=1}^{n}Y_i\leq u_n\right)=\exp\left\{-\lim_{n\to\infty} k_n P\left(\bigvee_{i=1}^{[n/k_n]}Y_i> u_n\right)\right\}
\end{eqnarray}
and
\begin{align}
&\displaystyle \lim_{n\to\infty} k_n P\left(\bigvee_{i=1}^{[n/k_n]}Y_i> u_n\right)
= \displaystyle \lim_{n\to\infty} k_n P\left(Y_1\leq u_n, \bigvee_{i=1}^{[n/k_n]}\{Y_i> u_n\}\right)\nonumber\\
= & \displaystyle \lim_{n\to\infty} k_n P\left(\bigcup_{i=1}^{[n/k_n]}\{Y_i\leq u_n<Y_{i+1}\}\right)
=\displaystyle \lim_{n\to\infty} k_n P\left(\bigcup_{i=1}^{[n/k_n]}\{Y_i\leq u_n<X_{i+1},U_{i+1}=1\}\right)\nonumber\\
=&\displaystyle \lim_{n\to\infty} k_n P\left(\bigcup_{i=1}^{[n/k_n]}\bigcup_{j=0}^{\kappa-1}\{X_{i-j}\leq u_n<X_{i+1},U_{i-j}=1,U_{i-j+1}=0=...=U_{i},U_{i+1}=1\}\right)\nonumber\\
=&\displaystyle \lim_{n\to\infty} k_n P\left(\bigcup_{i=1}^{[n/k_n]}\bigcup_{j=0}^{\kappa-1}\{X_{i}\leq u_n<X_{i+j+1},U_{i}=1,U_{i+1}=0=...=U_{i+j},U_{i+j+1}=1\}\right)\nonumber\\
=&\displaystyle \lim_{n\to\infty} k_n \sum_{i=1}^{[n/k_n]}\sum_{j=0}^{\kappa-1}P\left(X_1\leq u_n,...,X_{i}\leq u_n<X_{i+1},X_{i+j+1}>u_n\right)\cdot\nonumber\\
&\hspace{4cm}\cdot p_{i,i+1,...,i+j,i+j+1}(1,0,...,0,1)\label{pthetasExp1}\\
=&\displaystyle \lim_{n\to\infty} k_n \sum_{i=1}^{[n/k_n]}\sum_{j=0}^{\kappa-1}P\left(X_{i-s+2}\leq u_n,...,X_{i}\leq u_n<X_{i+1},X_{i+j+1}>u_n\right)\cdot\nonumber\\
&\hspace{4cm}\cdot p_{1,2,...,j+1,j+2}(1,0,...,0,1)\nonumber
\end{align}
since $\{X_n\}$ satisfies condition $D^{(s)}(u_n)$ for some $s\geq 2$. The stationarity of $\{X_n\}$ leads to
\begin{equation}\nonumber
\begin{array}{rl}
&\displaystyle \lim_{n\to\infty} k_n \sum_{i=1}^{[n/k_n]}\sum_{j=0}^{\kappa-1}P\left(X_{i-s+2}\leq u_n,...,X_{i}\leq u_<X_{i+1},X_{i+j+1}>u_n\right)\cdot\nonumber\\
&\hspace{4cm}\cdot p_{1,2,...,j+1,j+2}(1,0,...,0,1)\\
=&\displaystyle \lim_{n\to\infty} k_n \sum_{i=1}^{[n/k_n]}\sum_{j=0}^{\kappa-1}P\left(X_{1}\leq u_n,...,X_{s-1}\leq u_n<X_{s},X_{s+j}>u_n\right)\cdot\nonumber\\
&\hspace{4cm}\cdot p_{1,2,...,j+1,j+2}(1,0,...,0,1)\\
=&\displaystyle \lim_{n\to\infty} \sum_{j=0}^{\kappa-1}n\,P\left(X_{1}\leq u_n,...,X_{s-1}\leq u_n<X_{s},X_{s+j}>u_n\right)\cdot\nonumber\\
&\hspace{4cm}\cdot p_{1,2,...,j+1,j+2}(1,0,...,0,1)\\
=&\displaystyle \lim_{n\to\infty} \sum_{j=0}^{\kappa-1}n\,P\left(X_{1}\leq u_n,...,X_{s-1}\leq u_n<X_{s}\right)P\left(X_{s+j}>u_n|X_{1}\leq u_n,...,X_{s-1}\leq u_n<X_{s}\right)\cdot\nonumber\\
&\hspace{4cm}\cdot p_{1,2,...,j+1,j+2}(1,0,...,0,1)\\
=& \displaystyle \tau\, \theta_X \sum_{j=0}^{\kappa-1} p_{1,2,...,j+1,j+2}(1,0,...,0,1)\,\beta_j,
\end{array}
\end{equation}
where the last step follows from (\ref{thetaChernick}).
\end{proof}

Observe that $\sum_{j=0}^{\kappa-1} p_{1,2,...,j+1,j+2}(1,0,...,0,1)=p_n(1)=P(U_n=1)$ and thus $\theta_Y\leq \theta_Xp_n(1)\leq \theta_X$, as expected.

\begin{pro}
Suppose that $\{U_n\}$ is strong-mixing and $\{X_n\}$ satisfies conditions $D(u_n)$ and $D^{'}(u_n)$, for normalized levels $u_n\equiv u_n(\tau)$. Then $\{Y_n\}$ has extremal index $\theta_Y$ given by $\theta_Y=p_{1,2}(1,1)$.
\end{pro}
\begin{proof}
	By condition $D^{'}(u_n)$, the only term to consider in (\ref{pthetasExp1}) corresponds to $j=0$, and we obtain
	\begin{eqnarray}\nonumber
	\begin{array}{rl}
&\displaystyle\lim_{n\to\infty} k_n\,P\left(\bigvee_{i=1}^{[n/k_n]}Y_n\leq u_n\right)\\
=&\displaystyle\lim_{n\to\infty} k_n\,\sum_{i=1}^{[n/k_n]}P\left(X_1\leq u_n,...,X_{s-1}\leq u_n<X_s\right)p_{1,2}(1,1)\\
=&\displaystyle\lim_{n\to\infty} n\,P(X_s>u_n)\,p_{1,2}(1,1)=\tau \,p_{1,2}(1,1)\,.
	\end{array}
	\end{eqnarray}	
\end{proof}

Observe that we can obtain the above result by applying Proposition \ref{pD's} (iii) and calculating directly 
$\tau\, \theta_Y=\lim_{n\to\infty}n\, P(Y_1\leq u_n<Y_2)$. More precisely, we have that $\{Y_n\}$ satisfies $D^{(2)}(u_n)$ and by applying (\ref{thetaChernick}), we obtain
\begin{eqnarray}\nonumber
\begin{array}{rl}
\tau\, \theta_Y=&\displaystyle\lim_{n\to\infty} n\, P(Y_1\leq u_n<Y_2)\\
=&\displaystyle\lim_{n\to\infty} n\, P(Y_1\leq u_n<X_2,U_2=1)\\
=&\displaystyle\lim_{n\to\infty} nP\left(\bigcup_{j=0}^{\kappa-1}X_{1-j}\leq u_n<X_2,U_{1-j}=1,U_{1-j+1}=0=...=U_{1},U_2=1,\right)\\
=&\displaystyle\lim_{n\to\infty} nP\left(\bigcup_{j=0}^{\kappa-1}X_{2-\kappa}\leq u_n,...,X_{1-j}\leq u_n<X_{2-j},X_2>u_n\right)\cdot\\
&\hspace{4cm}\cdot p_{1-j,1-j+1,...,1,2}(1,0,...,0,1)\\
=&\displaystyle\lim_{n\to\infty} n\,P(X_1\leq u_n<X_2)\,p_{1,2}(1,1)\\
=&\displaystyle\lim_{n\to\infty} n\,P(X_2>u_n)\,p_{1,2}(1,1)=\tau \,p_{1,2}(1,1)\,.
\end{array}
\end{eqnarray}

The same result can also be seen as a particular case of Proposition \ref{pthetaYX} where, if we take $s=1$, we have $\beta_j=0$, for $j\not=0$, and we obtain $\theta_Y=\theta_X\beta_0 p_{1,2}(1,1)=p_{1,2}(1,1)$, since  $\beta_0=1$ and under $D^{'}(u_n)$ it comes $\theta_X=1$. \\

\begin{ex}\label{exARMAXtheta}
	Consider $\{Y_n\}$ such that $\{X_n\}$ is an ARMAX sequence, i.e., $X_n=\phi X_{n-1}\vee (1-\phi)Z_n$, $n\geq 1$, where $\{Z_n\}$ is an independent sequence of random variables with standard Fr\'echet marginal distribution and $\{X_n\}$ and $\{Z_n\}$ are independent. We have that $\{X_n\}$ has also standard Fr\'echet marginal distribution, satisfies condition $D^{(2)} (u_n)$ and has extremal index $\theta_X=1-\phi$ (see e.g.~Ferreira and Ferreira \cite{Ferreira+2012} 2012 and references therein). 
	
	Observe that, for normalized levels $u_n\equiv n/\tau$, $\tau>0$, we have
	\begin{eqnarray}\nonumber
	\begin{array}{rl}
	\beta_1=&{\displaystyle \lim_{n\to\infty}}P(X_3>u_n|X_1\leq u_n<X_2)\\
	=&{\displaystyle  \lim_{n\to\infty}}\frac{P(X_1\leq u_n)-P(X_1\leq u_n,X_2\leq u_n)-P(X_1\leq u_n,X_3\leq u_n)+P(X_1\leq u_n,X_2\leq u_n,X_3\leq u_n)}{P(X_1\leq u_n)-P(X_1\leq u_n,X_2\leq u_n)}\\
	=&{\displaystyle  \lim_{n\to\infty}}\frac{1-\frac{\tau}{n}-(1-\frac{\tau}{n}(2-\phi))-(1-\frac{\tau}{n}(2-\phi^2))+1-\frac{\tau}{n}(3-2\phi)}{1-\frac{\tau}{n}-(1-\frac{\tau}{n}(2-\phi))}\\
	=&\phi\,.
	\end{array}
	\end{eqnarray}
	Analogous calculations lead to $\beta_2=\phi^2$. Considering $\kappa=3$, we have $\theta_Y=(1-\phi)(p_{1,2}(1,1)+\phi p_{1,2,3}(1,0,1)+\phi^2p_{1,2,3,4}(1,0,0,1))$.
\end{ex}

The observed sequence is $\{Y_n\}$, therefore results that allow retrieving information about the extreme behavior of the initial sequence $\{X_n\}$, subject to the failures determined by $\{U_n\}$, may be of interest.

If we assume that $\{Y_n\}$ satisfies $D^{(s)}(u_n)$ then $\{X_n\}$ also satisfies $D^{(s)}(u_n)$ by Proposition \ref{pD's} (i), thus coming
\begin{eqnarray}\nonumber
\begin{array}{rl}
\tau\,\theta_X=&\displaystyle\lim_{n\to\infty} n\,P(X_1\leq u_n,...,X_{s-1}\leq u_n<X_s)\\
=&\displaystyle\lim_{n\to\infty} n\,P(Y_1\leq u_n,...,Y_{s-1}\leq u_n<Y_s|U_1=...=U_s=1)\\
=&\displaystyle\lim_{n\to\infty} n\,P(Y_1\leq u_n,...,Y_{s-1}\leq u_n<Y_s|Y_0\not=Y_1\not=...\not=Y_s).
\end{array}
\end{eqnarray}
Thereby, we can write
\begin{eqnarray}\nonumber
\begin{array}{rl}
\theta_X=&\displaystyle\lim_{n\to\infty} \frac{P(Y_1\leq u_n,...,Y_{s-1}\leq u_n<Y_s|Y_0\not=Y_1\not=...\not=Y_s)}{P(Y_1>u_n)}.
\end{array}
\end{eqnarray}

\section{Tail dependence}\label{stdc}

Now we will analyse the effect of this failure mechanism on the dependency between two variables, $Y_n$ and $Y_{n+m}$, $m\geq 1$. More precisely, we are going to evaluate the lag-$m$ tail dependence coefficient
\begin{eqnarray}\nonumber
\lambda(Y_{n+m}|Y_n)=\lim_{x\to\infty}P(Y_{n+m}>x|Y_n>x),
\end{eqnarray}
which incorporates the tail dependence between $X_n$ and $X_{n+j}$, with $j$ regulated by the maximum number of failures $\kappa-1$ and by the relation between $m$ and $\kappa$. In particular, independent variables present null tail dependence coefficients. If $m=1$ we obtain the tail dependence coefficient in Joe (\cite{Joe1997} 1997). For simplicity, we first present the case  $m = 1$ and then we extend the result to any value $m$.

\begin{pro}\label{plambdalag1}
	Sequence $\{Y_n\}$ has tail dependence coefficient
	\begin{eqnarray}\nonumber
	\lambda(Y_{n+1}|Y_n)=p_n(0)+\sum_{i=0}^{\kappa-1}\lambda(X_{n+1+i}|X_n)\,p_{1,2,...,i+1,i+2}(1,0,...,0,1),
	\end{eqnarray}
	provided all coefficients $\lambda(X_{n+1+i}|X_n)$ exist.
\end{pro}
\begin{proof} We have that
	\begin{eqnarray}\nonumber
	\begin{array}{rl}
&\displaystyle\lim_{x\to\infty}\frac{P(Y_n>x,Y_{n+1}>x)}{P(Y_{n}>x)}\\
=&\displaystyle\lim_{x\to\infty}\frac{P(Y_n>x,U_{n+1}=0)}{P(Y_{n}>x)}+\lim_{x\to\infty}\frac{P(Y_n>x,X_{n+1}>x,U_{n+1}=1)}{P(Y_{n}>x)}\\
=&\displaystyle\lim_{x\to\infty}\frac{\sum_{i=0}^{\kappa-2}P(X_{n-i}>x)\,p_{n-i,n-i+1,...,n+1}(1,0,...,0)}{P(Y_{n}>x)}\\
&+\displaystyle\lim_{x\to\infty}\frac{\sum_{i=0}^{\kappa-1}P(X_{n-i}>x,X_{n+1}>x)\, p_{n-i,n-i+1,...,n,n+1}(1,0,...,0,1)}{P(Y_{n}>x)}\\
=&\displaystyle \sum_{i=0}^{\kappa-2}p_{1,2,...,i+2}(1,0,...,0)+\sum_{i=0}^{\kappa-1}\lambda(X_{n+1+i}|X_n)\,p_{1,2,...,i+1,i+2}(1,0,...,0,1)\,.
	\end{array}
	\end{eqnarray}
\end{proof}

\begin{pro}\label{plambdalags}
	Sequence $\{Y_n\}$ has lag-$m$ tail dependence coefficient, with $m\geq 1$,
	\begin{eqnarray}\label{TDClag-m}
\begin{array}{rl}
	\displaystyle \lambda(Y_{n+m}|Y_n)=&p_{1,...,m}(0,...,0)\,\mathbf{1}_{\{m\leq \kappa-1\}}+\displaystyle \sum_{i=1\vee (m-\kappa+1)}^{m}\sum_{i^*=0}^{\kappa-1}\lambda(X_{n+i+i^*}|X_n)\cdot\\
	&\hspace{4cm}\displaystyle\cdot p_{1,2,...,i^*+1,i^*+1+i,i^*+2+i,...,i^*+1+m}(1,0,...,0,1,0,...,0),
\end{array}
	\end{eqnarray}
	provided all coefficients $\lambda(X_{n+i+i^*}|X_n)$ exist.
\end{pro}
\begin{proof} Observe that
	\begin{eqnarray}\nonumber
	\begin{array}{rl}
	&\displaystyle P(Y_n>x,Y_{n+m}>x)\\
	=&\displaystyle P(Y_n>x,U_{n+1}=0=...=U_{n+m})\mathbf{1}_{\{m\leq \kappa-1\}}\\
	&+\displaystyle\sum_{i=1\vee (m-\kappa+1)}^{m} P(Y_n>x,X_{n+i}>x,U_{n+i}=1,U_{n+i+1}=0=...=U_{n+m})\\
	=&\displaystyle \sum_{i=0}^{\kappa-1-m}P(X_{n-i}>x)\,p_{n-i,n-i+1,...,n+m}(1,0,...,0)\mathbf{1}_{\{m\leq \kappa-1\}}\\
	&+\displaystyle \sum_{i=1\vee (m-\kappa+1)}^{m}\sum_{i^*=0}^{\kappa-1} P(X_{n-i^*}>x,X_{n+i}>x)\, p_{n-i^*,n-i^*+1,...,n,n+i,n+i+1,...,n+m}(1,0,...,0,1,0,...,0)\,
	\end{array}
	\end{eqnarray}
and 
$\displaystyle \sum_{i=0}^{\kappa-1-m}p_{1,2,...,m+i+1}(1,0,...,0)=p_{1,...,m}(0,...,0).$
\end{proof}
Taking $m=1$ in (\ref{TDClag-m}), we immediately obtain the result of Proposition \ref{plambdalag1}. 

If $\{X_n\}$ is lag-$m^*$ tail independent for all integer $m^*\geq 1\vee(m-\kappa+1)$, we have $\lambda(X_{n+i+i^*}|X_n)=0$ in the second ter of (\ref{TDClag-m}) and thus $\lambda(Y_{n+m}|Y_n)=p_{1,...,m}(0,...,0)\,\mathbf{1}_{\{m\leq \kappa-1\}}$ and $\{Y_n\}$ is lag-$m$ tail independent for all integer $m\geq \kappa$.

\begin{ex}\label{exARMAXTDClag-m}
	Consider again $\{Y_n\}$ based on ARMAX sequence $\{X_n\}$ as in Example \ref{exARMAXtheta}. We have that $\{X_n\}$ has lag-$m$ tail dependence coefficient $\lambda(X_{n+m}|X_n)=\phi^m$ (Ferreira and Ferreira \cite{Ferreira+2012} 2012) and thus 
	\begin{eqnarray}\nonumber
	\begin{array}{rl}
	\displaystyle \lambda(Y_{n+m}|Y_n)=&p_{1,...,m}(0,...,0)\,\mathbf{1}_{\{m\leq \kappa-1\}}\\
	&+\displaystyle \sum_{i=1\vee (m-\kappa+1)}^{m}\sum_{i^*=0}^{\kappa-1}\phi ^{i+i^*} p_{1,2,...,i^*+1,i^*+1+i,i^*+2+i,...,i^*+1+m}(1,0,...,0,1,0,...,0).
	\end{array}
	\end{eqnarray}
\end{ex}


\bigskip

\begin{thebibliography}{}
	
\bibitem{Chernick+1991} Chernick M.R., Hsing T., McCormick W.P. (1991). Calculating the extremal index for a class of stationary sequences. Adv. Appl. Probab. 23, 835--850.
	
\bibitem{Ferreira+2012} Ferreira, M., Ferreira, H. (2012). On extremal dependence: some contributions. TEST 21(3), 566--583.
	
\bibitem{FerreiraH++2019} Ferreira, H. Martins, A.P., Temido, M.G. (2019). Extremal behaviour of a periodically controlled sequence with imputed values. arXiv:1907.11336 (submitted).

\bibitem{HallHusler2006} Hall, A. and H\"usler, J. (2006). Extremes of stationary sequences with failures. Stoch. Models. 22, 537--557.

\bibitem{Joe1997} Joe, H. (1997) \textit{Multivariate Models and Dependence Concepts}. Monographs on Statistics and Applied Probability 73, Chapman and Hall, London.

\bibitem{Leadbetter1974}  Leadbetter M.R. (1974). On extreme values in stationary sequences. Z. Wahrscheinlichkeitstheor Verw. Geb. 28(4), 289--303.

\bibitem{Leadbetter+1983}  Leadbetter, M.R., Lindgren, G. and Rootz\'en, H. (1983). Extremes and Related Properties of Random Sequences and Processes. New York: Springer-Verlag.

\bibitem{Rosenblat1956} Rosenblatt M. (1956). A central limit theorem and a strong mixing condition. Proceedings of the National Academy of Sciences of the United States of America, 42(1), 43--47. 


\end{thebibliography}
\end{document}